\documentclass{amsproc}

\usepackage{amsmath}
\usepackage{amssymb}
\usepackage{rotating}
\usepackage{graphicx}
\usepackage[tableposition=below]{caption}
\usepackage{bm}

\DeclareGraphicsExtensions{.png,.pdf,.eps}
\usepackage[all,2cell,cmtip]{xy}
\usepackage{tikz}
\usepackage{xcolor}
\usepackage{longtable}
\usepackage{soul}
\newtheorem{theorem}{Theorem}[section]
\newtheorem{lemma}[theorem]{Lemma}

\newtheorem{proposition}[theorem]{Proposition}

\newtheorem{claim}[theorem]{Claim}
\theoremstyle{definition}

\newtheorem{definition}[theorem]{Definition}

\newtheorem{remark}[theorem]{Remark}

\def\P{{\mathbb P}}

\def\R{{\mathbb R}}
\def\Z{{\mathbb Z}}

\def\cE{{\mathcal E}}
\def\cF{{\mathcal F}}

\def\cL{{\mathcal L}}

\def\cO{{\mathcal{O}}}

\def\ra{\rightarrow}

\def\tensor{\otimes}

\def\cOperatorname#1{\mathop{\rm #1}\nolimits}

\def\NE{{\cOperatorname{NE}}}
\def\ME{{\cOperatorname{ME}}}

\newcommand{\cNE}{\overline{\NE}}
\newcommand{\cME}[1]{\cOverline{\ME}}

\makeindex

\begin{document}

\title{Weak Fano bundles of rank $2$ over hyperquadrics $Q^n$ of dimension $n \ge 5$}

\author{Yuta Takahashi}
\date{\today}
\address{Department of Mathematics, Faculty of Science and Engineering, Chuo University.
1-13-27 Kasuga, Bunkyo-ku, Tokyo 112-8551, Japan}
\email{yuta0630takahashi0302@gmail.com}
\subjclass[2010]{14J40 14J45, 14J60, 14J70.}
\keywords{Weak Fano bundle, Quadric hypersurfaces, Rank $2$ vector bundle}
\maketitle
\begin{abstract} 
A vector bundle whose projectivization becomes a weak Fano variety is called a weak Fano bundle. We present classification results for rank 2 weak Fano bundles on higher-dimensional quadrics $Q^n$ of dimension $\ge 5$. 
\end{abstract}
\tableofcontents
\section{Introduction}In this paper, we study weak Fano bundles over a smooth quadric hypersurface $Q^n$ defined over an algebraically closed field $k$ of characteristic zero. 
A vector bundle $\cE$ on a smooth projective variety $X$ is a {\it Fano bundle} if the projectivization of $\cE$ is a Fano variety. Szurek and Wi{\'s}niewski showed that the base space $X$ of a Fano bundle $\cE$ must be a smooth Fano variety \cite[Theorem~1.6]{SW90}. They also gave a classification of rank $2$ Fano bundles on $\P^3$ and studied  rank $2$ Fano bundles on $Q^3$ \cite[Theorem~2.1 and Section~3]{SW90}. Subsequently, rank 2 Fano bundles on $\P^n$ and $Q^n$ in arbitrary dimensions are classified in \cite[Main Theorem]{APW94}. 

After that, Sols, Szurek and Wi{\'s}niewski completed the classification of rank $2$ Fano bundles on $Q^3$ in \cite{SSW91}. More generally, Mu{\~n}oz, Occhetta and Sol{\'a}~Conde completely classified rank $2$ Fano bundles on smooth Fano varieties $X$ with $H^2(X,\Z)\simeq H^4(X,\Z)\simeq \Z$ in \cite{MOS12, MOS14, MOS142}.  

On the other hand, Langer introduced a {\it weak Fano bundle} as a natural generalization of Fano bundles in \cite{L98}. A vector bundle $\cE$ on a smooth projective variety $X$ is a {\it weak Fano bundle} if the projectivization of $\cE$ is a weak Fano variety.  Until now rank $2$ weak Fano bundles were classified when the base space is the projective space $\P^n$ \cite{Y12} or a cubic hypersurface in $\P^4$, i.e., a del Pezzo 3-fold of degree three \cite{I16} or a del Pezzo 3-fold $X$ of deg$X$ $\in \{1,2,4,5\}$ \cite{FHI22, FHI23}. In addition, weak Fano bundles of arbitrary rank $r$ whose first Chern class is odd over $\P^2$ are classified, see \cite{L98, O20}. 

As opposed to Fano bundles, rank $2$ weak Fano bundles on quadric hypersurfaces $Q^n$ where $n\ge 3$ have not yet been classified. Our main purpose of this paper is to prove the following theorem.
\begin{theorem}\label{thm:main2}
Let $\cE$ be a weak Fano bundle of rank $2$ over $Q^n$ where $n\ge 5$. Then $\cE$ is a direct sum of line bundles or the Cayley bundle on $Q^5$ up to twist with a line bundle (see Definition~\ref{def:cayley}).
\end{theorem}

To prove this theorem, we follow the idea of \cite{APW94}. The approach to classify Fano bundles of rank $2$ on $Q^n$ in \cite{APW94} is based on two key ingredients. The first ingredient is to show that $\cE(n-1)$ is globally generated for a normalized Fano bundle $\cE$ of rank $2$ on $Q^n$ of $n\ge 5$. In this paper, we prove that an analogous result also holds for a normalized weak Fano bundle of rank $2$. The second ingredient of \cite{APW94} is to establish splitting criterions for vector bundles on $Q^n$. We also use these criterions to prove Theorem~\ref{thm:main2}. Although the basic idea of our proof of Theorem~\ref{thm:main2} is similar to that of \cite{APW94}, the calculations are more complicated and our case requires more discussion than \cite{APW94}.

\subsection*{Notations}Throughout this part, $k$ is an algebraically closed field of characteristic zero. We use standard notations as in \cite{H77, L041, L042, OSS80}. Let $X$ be a smooth projective variety and $\cE$ be a rank $2$ vector bundle over $X$. 
\begin{itemize}
\item A Cartier divisor $D$ on $X$ is {\it nef} if for all irreducible curves $C\subset X$, $D\cdot C\ge 0$.
\item A vector bundle $\cE$ over $X$ is called {\it nef} if the tautological line bundle  $\cO_{\P(\cE)}(1)$ is nef.
\item We say that $X$ is {\it Fano} if $-K_X$ is ample. 
\item We say that $X$ is {\it weak Fano} if $-K_X$ is nef and big.  
\item A vector bundle $\cE$ is {\it stable (resp. semistable)} if for every invertible subsheaf $\cL$ of $\cE$, $c_1(\cL)<\frac{1}{2}c_1(\cE)$ (resp. $c_1(\cL)\le \frac{1}{2}c_1(\cE)$).
\item A vector bundle $\cE$ {\it splits} if $\cE$ is a direct sum of line bundles. 
\item We denote by $A_k(X)=A^{n-k}(X)$ the group of rational equivalence classes of algebraic $k$-cycles on $X$. We denote by  $A(X):=\bigoplus_k A_k(X)$ the Chow ring of $X$.
\item We denote the $i$-th Chern class of $\cE$ by $c_i(\cE)$. When $A^1(X)$ and $ A^2(X)$ are isomorphic to $\Z$, there exist an effective divisor $H$ and an effective $2$-cocycle $L$ on $X$ such that $A^1(X)\cong \Z[H]$ and $A^2(X)\cong \Z[L]$; then we consider $c_1(\cE)$ and $c_2(\cE)$ as integers $c_1$ and $c_2$, that is, $c_1(\cE)=c_1H\in A^1(X)$ and $c_2(\cE)=c_2L\in A^2(X)$.
\item Assume that Pic$(X)\simeq \Z$. A vector bundle $\cE$ is {\it normalized} if $c_1=0$ or $-1$.
\item We denote the $i$-th Segre class of $\cE$ by $s_i(\cE)$. It is defined by the equation $s_i(\cE)=\pi_{*}(\xi_{\cE}^{1+i})$, where $\pi:\P(\cE)\rightarrow X$ is the natural projection and $\xi_{\cE}$ is the tautological divisor corresponding to $\cO_{\P(\cE)}(1)$. 
\item We denote by $\P^n$ and $Q^n$ projective $n$-space and a smooth quadric hypersurface in $\P^{n+1}$ respectively.
\end{itemize}
\section{Preliminaries}
\subsection{Chow ring of quadric hypersurfaces $Q^n$}

For $n\ge 3$, let $H_{Q^n} \in A^1(Q^n)$ be the class of the hyperplane section. Since $n\ge 3$, $A^1(Q^n)\simeq {\rm Pic}(Q^n)$ is generated by $H_{Q^n}$. If $n=2k+1$, then we denote $P_{Q^n}\in A^k(Q^{2k+1})$ by the generator of it.

\begin{proposition}[{\cite[Section~1]{F83}}]\label{prop:chow}
Assume that $n=2k+1$ for $k\ge 1$. Then the classes $H_{Q^n}$ and $P_{Q^n}$ satisfy the following relations:
\begin{itemize}
\item[(i)]$H_{Q^n}^{k+1}=2P_{Q^n}$,
\item[(ii)]$H^k_{Q}\cdot P_{Q^n}=1$.
\end{itemize}
\end{proposition}

\subsection{Some properties of vector  bundles}

\begin{proposition}\label{prop:normsta}
For a normalized vector bundle $\cE$ of rank $2$ over a smooth projective variety $X$ with ${\rm Pic}(X)\simeq \Z$, the following hold : \\
(1) $\cE$ is stable if and only if $\cE$ has no sections:
$$
H^0(X, \cE)=0.
$$
(2) If $c_1(\cE)$ is even, then $\cE$ is semistable if and only if 
$$
H^0(X,\cE(-1))=0.
$$
\end{proposition}

\begin{proof}
This can be proved by the same argment as in \cite[Lemma~1.2.5]{OSS80}.
\end{proof}

\begin{proposition}[{\cite[Corollary~3.5]{MP97}}]\label{prop:bogo}
Let $X$ be a normal projective variety over $k$ 
 , smooth in codimension two. If $\cE$ is a semistable  torsion free sheaf of rank $2$, then
$$
c_1^2(\cE)\le 4c_2(\cE).
$$
If $\cE$ is a stable  torsion free sheaf of rank $2$, then
$$
c_1^2(\cE)< 4c_2(\cE).
$$
\end{proposition}
We recall the Castelnouvo-Mumford regularity of a coherent sheaf, which measures its cohomological complexities. 
\begin{definition}\label{def:mreg}
Let $X$ be a projective variety and $B$ an ample and globally generated line bundle on $X$. A coherent sheaf $\cF$ is $m$-regular with respect to $B$ if $H^i(X, \cF\tensor B^{\tensor m-i})=0$ for $i>0$.
\end{definition}

The following theorem is used for Proposition~\ref{prop:gbgs}. 
\begin{theorem}[{\cite[Theorem~1.8.5]{L041}}]\label{thm:mreg}
Let $X$ be a projective variety and $\cF$ an $m$-regular with respect to an ample line bundle $B$. Then for all $k\ge 0$, 
\begin{enumerate}
\item $\cF\tensor B^{\tensor m+k}$ is globally generated, 
\item $H^0(X, \cF\tensor B^{\tensor m}) \tensor H^0(X, B^{\tensor k}) \ra H^0(X, \cF\tensor B^{\tensor m+k})$ is surjective, 
\item $\cF$ is $(m+k)$-regular with respect to $B$.
\end{enumerate}
\end{theorem}
To prove Theorem~\ref{thm:main2}, we use the following vanishing theorems.
\begin{theorem}[{\cite[Theorem~1]{V82}}]\label{thm:KV-van}
Let $X$ be a smooth projective variety and $\cL$ a line bundle over $X$. If $\cL$ is nef and big, then $H^i(X, \cL\tensor K_X)=0$ for $i >0$.
\end{theorem}

\begin{theorem}[{\cite[Theorem~1]{L75}}]\label{thm:le}
Let $X$ be a smooth projective variety and $\cE$ an ample vector bundle of rank $r$ over $X$. Then $H^i(X, \cE\tensor K_X)=0$ for $i \ge r$.
\end{theorem}

\subsection{Splitting criterions and numerical properties}

The following splitting criterions play key roles in the proof of Theorem~\ref{thm:main2}. 
\begin{proposition}[{\cite[Corollary~4.3]{APW94}}]\label{lem:unsta-chern}
Let $\cF$ be a globally generated bundle of rank $2$ on $Q^n$ where $n\ge 5$. If $\cF$ is unstable and 
$$
c_2(\cF) \le (n-2)(c_1(\cF)-n+2)+n-3, 
$$
then $\cF$ splits.
\end{proposition}

\begin{proposition}[{\cite[Proposition~6.2]{APW94}}]\label{lem:12sin}
Let $\cF$ be a vector bundle of rank $2$ on $Q^n$ where $n\ge 12$, admitting a section whose zero locus is smooth and of pure codimension two. Assume $c_1(\cF)^2<4c_2(\cF)$. Then
$$
c_2(\cF)>\displaystyle\frac{71}{4{\rm sin}^2\left(\displaystyle\frac{\pi}{n-1}\right)}.
$$
\end{proposition}

Following lemmas are used in the next sections. For calculations, we refer to Proposition~\ref{prop:chow}, \cite[Theorem~4.1 of Appendix~A]{H77} and \cite{F98}.

\begin{lemma}\label{lem:5RR}
Let $\cE$ be a vector bundle of rank $2$ on $Q^5$. Then
$$
\chi(\cE)= 2 +\frac{894}{360}c_1 +\frac{55}{24}(c_1^2
 -2c_2)+c_1^3-3c_1c_2+\frac{5}{24}(c_1^4-4c_1^2c_2+2c_2^2)+\frac{1}{60}(c_1^5-5c_1^3c_2+5c_1c_2^2).
$$
\end{lemma}

\begin{lemma}\label{lem:segre}
On $Q^6$
$$
s_6=2c_1^6-10c_1^4c_2+12c_1^2c_2^2-2c_2^3.
$$
\end{lemma}

\begin{proof}
By $s_t(\cE)c_{-t}(\cE)$=1, $s_1(\cE)=c_1(\cE)$, $s_2(\cE)=c_1^2(\cE)-c_2(\cE)$, $s_3(\cE)=c_1(\cE)(c_1^2(\cE)-2c_2(\cE))$, $s_4(\cE)=c_1^4(\cE)-3c_1^2(\cE)c_2(\cE)+c_2^2(\cE)$ and $s_5(\cE)=c_1(\cE)(c_1^2(\cE)-3c_2(\cE))(c_1^2(\cE)-c_2(\cE))$, we have $s_6=2c_1^6-10c_1^4c_2+12c_1^2c_2^2-2c_2^3$ on $Q^6$.
\end{proof}

Finally, we show the relation between normalized weak Fano bundles of rank $2$ and nef vector bundles.
\begin{lemma}\label{lem:n-c1-nef}
Let $\cE$ be a normalized weak Fano bundle of rank $2$ on $Q^n$. Then $\cE\left(\frac{n-c_1(\cE)}{2}\right)$ is nef.
\end{lemma}

\begin{proof}
Let $\cE_0:= \cE\left(\frac{n-c_1(\cE)}{2}\right)$. We show that $\xi_{\mathbb{P}(\cE_0)}$ is nef. Since 
\begin{align*}
-K_{\mathbb{P}(\cE)}&=-K_{\mathbb{P}\left(\cE_0\right)}\\
&= 2\xi_{\cE_0}+\left(n-\left(c_1(\cE)+2\cdot \frac{n-c_1(\cE)}{2}\right)\right)H\\
&= 2\xi_{\cE_0}
\end{align*}
and $\cE$ is a weak Fano bundle, $\xi_{\P\left(\cE_0\right)}$ is nef.
\end{proof}

\section{Globally generated property}

In this section, we will show the following proposition, which is a generalization of \cite[Proposition~2.7]{APW94}. 
\begin{proposition}\label{prop:gbgs}
Let $\cE$ be a normalized weak Fano bundle of rank $2$ on $Q^n$ where $n\ge 5$. Then $\cE(n-1)$ is globally generated.
\end{proposition}

\begin{proof}
For $X=\P(\cE)$, we denote a natural projection $X\rightarrow Q^n$ by $\pi$ and set $H=\pi^*(\cO_{Q^n}(1))$. We consider by dividing into three cases.

Assume that $n$ is even. For $n=2k$ $(k\ge 3)$, we have
\begin{align*}
-K_X&=2\xi_{\cE}+(2k-c_1(\cE))H\\
&=2\xi_{\cE(k+1)}-(c_1(\cE)+2)H.
\end{align*}

\begin{claim}
$\cE(k+1)$ is ample.
\end{claim}

\begin{proof}
Since $X$ is a weak Fano variety, there exist rational curves $C_1$ and $C_2$ such that $\cNE(X)=\R_{\ge0}[C_1]+\R_{\ge0}[C_2]$ by \cite[Theorem~1.4(2)]{Y12}. Let $F$ be a fiber of $\pi$ and $L$ a line contained  in $F$. Then
\begin{align*}
(-K_X+(c_1(\cE)+2)H)\cdot L&=2\xi_{\cE(k+1)}\cdot L\\
&=2 > 0.
\end{align*}
Let $C$ be a rational curve which is not contracted by $\pi$. Then
\begin{align*}
(-K_X+(c_1(\cE)+2)H)\cdot C&=-K_X\cdot C+(c_1(\cE)+2)H\cdot C\\
&\ge(c_1(\cE)+2)H\cdot C\\
&=(c_1(\cE)+2)\cO_{Q^n}(1)\cdot \pi_*C\\
&>0.
\end{align*}
Therefore by Kleiman's criterion, $2\xi_{\cE(k+1)}$ is ample. Hence $\xi_{\cE(k+1)}$ is ample.
\end{proof}
For $n \ge i\ge 2$, we put $j=3k-i-2$. Then we have 
$$
j=3k-i-2\ge 3k-n-2=k-2\ge 0.
$$
By Theorem~\ref{thm:le}, we obtain 
$$
H^i(Q^n, \cE(j-k+1))=H^i(Q^n, \cE(k+1+j)\tensor K_{Q^n})=0.
$$
Thus
$$
H^i(Q^n, \cE(n-1-i))=0
$$
for $i\ge 2$.  
We show that this holds for \(i=1\). Since \(n-2=2k-2\geq k+1\),
the ampleness of \(\mathcal{E}(k+1)\) implies that
\(\mathcal{E}(n-2)\) is ample. Hence \(\xi_{\mathcal{E}(n-2)}\) is ample on
\(X=\mathbb{P}(\mathcal{E})\). Put
$$
L:=\xi_{\mathcal{E}(n-2)}-K_X.
$$
Then \(L=\xi_{\mathcal{E}(n-2)}+(-K_X)\) is nef and big, since
\(\xi_{\mathcal{E}(n-2)}\) is ample and \(-K_X\) is nef and big.
Therefore, by Theorem 2.6, we obtain
$$
H^1(Q^n,\mathcal{E}(n-2))
=
H^1(X,\xi_{\mathcal{E}(n-2)})
=
H^1(X,K_X+L)
=0.
$$

Therefore 
$$
H^i(Q^n, \cE(n-1-i))=0
$$
for $i\ge 1$. Namely, $\cE$ is $(n-1)$-regular with respect to $\cO_{Q^n}(1)$. By Theorem~\ref{thm:mreg} (i), $\cE(n-1)$ is globally generated.

Assume that $n$ is odd and $c_1(\cE)=-1$. For  $n=2k+1$ $(k\ge2)$, we have $-K_X=2\xi_{\cE}+(2k+2)H=2\xi_{\cE(k+2)}-2H$. Then $\cE(k+2)$ is ample by  the similar   argument of Claim~3.2. 
 
We show that this holds for \(i=1\). In this case, we have
$$
-K_X=2\xi_{\mathcal{E}}+(2k+2)H=2\xi_{\mathcal{E}(k+1)}.
$$
Since \(X\) is weak Fano, \(-K_X\) is nef and big. Hence
\(\xi_{\mathcal{E}(k+1)}\) is nef and big. Moreover,
$$
\xi_{\mathcal{E}(n-2)}
=
\xi_{\mathcal{E}(k+1)}+(k-2)H.
$$
Put
$$
L:=\xi_{\mathcal{E}(n-2)}-K_X.
$$
Then
$$
L
=
3\xi_{\mathcal{E}(k+1)}+(k-2)H.
$$
Since \(k\geq 2\), the divisor \(L\) is nef and big. Therefore, by
Theorem 2.6, we obtain
$$
H^1(Q^n,\mathcal{E}(n-2))
=
H^1(X,\xi_{\mathcal{E}(n-2)})
=
H^1(X,K_X+L)
=0.
$$

Therefore 
$$
H^i(Q^n, \cE(n-1-i))=0
$$
for $i\ge 1$. Namely, $\cE$ is $(n-1)$-regular with respect to $\cO_{Q^n}(1)$. By Theorem~\ref{thm:mreg} (i), $\cE(n-1)$ is globally generated. 

Assume that $n$ is odd and $c_1(\cE)=0$. For $n=2k+1$ $(k\ge2)$, we have $-K_X=2\xi_{\cE}+(2k+1)H=2\xi_{\cE(k+1)}-H$. Then $\cE(k+1)$ is ample by  the similar  argument of Claim~3.2. 
 
We show that this holds for \(i=1\). Since \(n-2=2k-1\geq k+1\),
the ampleness of \(\mathcal{E}(k+1)\) implies that
\(\mathcal{E}(n-2)\) is ample. Hence \(\xi_{\mathcal{E}(n-2)}\) is ample on
\(X=\mathbb{P}(\mathcal{E})\). Put
\[
L:=\xi_{\mathcal{E}(n-2)}-K_X.
\]
Then \(L=\xi_{\mathcal{E}(n-2)}+(-K_X)\) is nef and big, since
\(\xi_{\mathcal{E}(n-2)}\) is ample and \(-K_X\) is nef and big.
Therefore, by Theorem 2.6, we obtain
\[
H^1(Q^n,\mathcal{E}(n-2))
=
H^1(X,\xi_{\mathcal{E}(n-2)})
=
H^1(X,K_X+L)
=0.
\]

Therefore 
$$
H^i(Q^n, \cE(n-1-i))=0
$$
for $i\ge 1$. Namely, $\cE$ is $(n-1)$-regular with respect to $\cO_{Q^n}(1)$. By Theorem~\ref{thm:mreg} (i), $\cE(n-1)$ is globally generated. 
\end{proof}

\section{Weak Fano bundle of rank $2$ over $Q^n$ of dimension $n\ge 5$}

\subsection{The case $c_1(\cE)^2\ge 4c_2(\cE)$}
\begin{proposition}\label{prop:unsta-sp}
Let $\cE$ be a weak Fano bundle of rank $2$ over $Q^n$ where $n\ge 5$.  If $c_1(\cE)^2\ge 4c_2(\cE)$, then $\cE$ splits.
\end{proposition}

\begin{proof}
Assume that $\cE$ is normalized. Since $c_1(\cE)^2\ge 4c_2(\cE)$, $\cE$ is unstable by Proposition~\ref{prop:bogo}. By Proposition~\ref{prop:gbgs}, $\cF:=\cE(n-1)$ is globally generated. Then we obtain
\begin{align*}
c_2(\cF)&=c_2(\cE)+(n-1)c_1(\cE)+(n-1)^2\\
&< (n-1)c_1(\cE)+(n-1)^2+1\\
&=(n-1)(c_1(\cF)-2(n-1))+(n-1)^2+1\\
&=(n-1)c_1(\cF)-(n-1)^2+1\\
&=(n-1)c_1(\cF)-n^2+2n+c_1(\cF)-c_1(\cF)\\
&=(n-2)c_1(\cF)-n^2+2n+2(n-1)+c_1(\cE)\\
&=(n-2)(c_1(\cF)-n+2)+c_1(\cE)+2+(n-3)-(n-3)\\
&= (n-2)(c_1(\cF)-n+2)+n-3 -n+5+c_1(\cE)\\
&\le (n-2)(c_1(\cF)-n+2)+n-3 +c_1(\cE)\\
&\le (n-2)(c_1(\cF)-n+2)+n-3.
\end{align*}
By Proposition~\ref{lem:unsta-chern}, $\cF$ splits. Therefore $\cE$ splits.
\end{proof}

\subsection{The case $c_1(\cE)^2< 4c_2(\cE)$}

\begin{proposition}\label{prop:nexist}
There are no rank $2$ weak Fano bundles $\cE$ with $c_1(\cE)^2< 4c_2(\cE)$ over $Q^n$ where $n\ge 12$.
\end{proposition}

\begin{proof}
Assume that $\cE$ is normalized. Set $\cF:=\cE(n-1)$. Then we have $c_1(\cF)=c_1(\cE)+2(n-1)\le 2(n-1)$. By Proposition~\ref{prop:gbgs}, $\cF$ is globally generated. Thus by \cite[Proposition~1.4(A)(3)]{APW94}, we have $c_1(\cF)^2\ge 3c_2(\cF)$. Then $c_2(\cF)\le\frac{4}{3}(n-1)^2$. In addition, $\cF$ admits a section whose zero locus is codimension two by \cite[Remark~1.1.1]{H78} and smooth by the same argument as in \cite[Proposition~1.4(b)]{H78}. Since $c_1(\cF)^2< 4c_2(\cF)$, we have $c_2(\cF)>\displaystyle\frac{71}{4{\rm sin}^2\left(\displaystyle\frac{\pi}{n-1}\right)}$ by Proposition~\ref{lem:12sin}. In order to get a contradiction, we show that  
$$
\frac{4}{3}(n-1)^2 <\frac{71}{4{\rm sin}^2\left(\displaystyle\frac{\pi}{n-1}\right)} \hspace{2.0mm}{\rm for}\hspace{1.0mm} n\ge 12.
$$ 
Write $m:=n-1$. It is replaced by $\displaystyle\frac{4}{3}m^2 <\displaystyle\frac{71}{4{\rm sin}^2\left(\displaystyle\frac{\pi}{m}\right)}$. In other words, it is enough to show that 
$$
\displaystyle\frac{2}{\sqrt{3}}m <\displaystyle\frac{\sqrt{71}}{2{\rm sin}\left(\displaystyle\frac{\pi}{m}\right)} \hspace{2.0mm}{\rm for}\hspace{1.0mm} m\ge 11.
$$
 
Set
\[
f(x):=\frac{\sqrt{213}}{4}-\frac{1}{x}\sin(\pi x)
\quad
\left(0<x\leq \frac{1}{11}\right).
\]
We show that \(f(x)>0\) on this interval. A direct calculation gives
\[
f'(x)
=
\frac{\sin(\pi x)-\pi x\cos(\pi x)}{x^2}.
\]
Put
\[
g(x):=\sin(\pi x)-\pi x\cos(\pi x).
\]
Then
\[
g'(x)=\pi^2x\sin(\pi x)>0
\]
for \(0<x\leq 1/11\), and \(\displaystyle\lim_{x\to +0}g(x)=0\). Hence
\(g(x)>0\) for \(0<x\leq 1/11\). Therefore \(f'(x)>0\), and so
\(f(x)\) is increasing. Moreover,
\[
\lim_{x\to +0}f(x)
=
\frac{\sqrt{213}}{4}-\pi
>0.
\]
Thus \(f(x)>0\) for every \(0<x\leq 1/11\).
In other words, $\displaystyle\frac{4}{3}(n-1)^2 <\displaystyle\frac{71}{4{\rm sin}^2\left(\displaystyle\frac{\pi}{n-1}\right)}$ for $n\ge 12$. This is a contradiction.
\end{proof}

\begin{definition}\label{def:cayley}
The Cayley bundle $\cE$ over $Q^5$ is defined as a stable vector bundle of rank $2$ with Chern classes $(c_1(\cE), c_2(\cE))=(-1,1)$.
\end{definition}

\begin{proposition}\label{prop:5-11}
Let $\cE$ be a weak Fano bundle of rank $2$ over $Q^n$ where $5\le n\le11$. If $c_1(\cE)^2< 4c_2(\cE)$, then $n=5$ and $\cE$ is a Cayley bundle up to a twist. 
\end{proposition}

\begin{proof}
Assume that $\cE$ is normalized. Write $\cE_0:=\cE\left(\frac{n-c_1(\cE)}{2}\right)$. By Lemma~\ref{lem:n-c1-nef}, $\cE_0$ is nef. Applying \cite[Proposition~1.4]{APW94}, we have $c_1(\cE_0)^2\ge\alpha c_2(\cE_0)$ with 
$$
\alpha=
\begin{cases}
3&(n=5,6)\\
2+\sqrt{2}&(n=7,8)\\
\frac{5}{2}+\frac{\sqrt{5}}{2}&(n=9,10)\\
2+\sqrt{3}&(n\ge 11).
\end{cases}
$$
Since  $c_1(\cE_0)^2=(c_1(\cE)+(n-c_1(\cE)))^2=n^2$ and
$$
\alpha c_2(\cE_0)=\alpha\left(c_2\left(\cE\right)-\frac{1}{4}c_1(\cE)^2+\frac{1}{4}n^2\right), 
$$
we have $n^2\ge\alpha\left(c_2\left(\cE\right)-\frac{1}{4}c_1(\cE)^2+\frac{1}{4}n^2\right)$.

\begin{claim}
$c_2(\cE)\le 2$, if $n\neq6$ and $c_2(\cE)\le 3$, if $n=6$.
\end{claim}

\begin{proof}
Consider the case where $n=5$ and $c_1(\cE)=0$. Since $3\left(c_2\left(\cE\right)+\displaystyle\frac{25}{4}\right)\le 25$, we have $c_2(\cE)\le \displaystyle\frac{25}{12}$. Therefore $c_2(\cE)\le 2$. Other cases can be calculated similarly.
\end{proof}

Firstly we consider the case of $c_2=3$. 
 
If \(n=6\) and \(c_1=-1\), then
\[
c_1\left(\mathcal{E}\left(\frac{7}{2}\right)\right)=6,
\qquad
c_2\left(\mathcal{E}\left(\frac{7}{2}\right)\right)=\frac{47}{4}.
\]
Thus Lemma 2.11 gives
\[
\begin{aligned}
s_6\left(\mathcal{E}\left(\frac{7}{2}\right)\right)
&=
2c_1^6\left(\mathcal{E}\left(\frac{7}{2}\right)\right)
-10c_1^4\left(\mathcal{E}\left(\frac{7}{2}\right)\right)
c_2\left(\mathcal{E}\left(\frac{7}{2}\right)\right) \\
&\quad
+12c_1^2\left(\mathcal{E}\left(\frac{7}{2}\right)\right)
c_2^2\left(\mathcal{E}\left(\frac{7}{2}\right)\right)
-2c_2^3\left(\mathcal{E}\left(\frac{7}{2}\right)\right) \\
&=
\frac{1}{32}(-82223)
<0.
\end{aligned}
\]

This contradicts for Lemma~\ref{lem:n-c1-nef}. Next we consider the case of $c_2=2$. Take $Q^5 \subset Q^n$. If $c_1=0$,  we have 
\begin{align*}
\chi(\cE|_{Q^5})&=2+\frac{55}{24}(-2)+\frac{5}{24}\cdot 8\\
&= -\frac{11}{2}\notin \Z
\end{align*}
by Lemma~\ref{lem:5RR}. Similarly, if $c_1=-1$, we have
\begin{align*}
\chi(\cE|_{Q^5})&=2 -\frac{894}{360} -\frac{55}{24}-1+6+\frac{5}{24}(1-8+8)+\frac{1}{60}(-1+10-20)\\
&= -\frac{7}{3}\notin \Z
\end{align*}
by Lemma~\ref{lem:5RR}. Next we consider the case of $c_1=0$ and $c_2=1$. We have
\begin{align*}
\chi(\cE|_{Q^5})&=2 -\frac{55}{24}+\frac{5}{12}\\
&= -\frac{13}{6}\notin \Z
\end{align*}
by Lemma~\ref{lem:5RR}. It remains to consider the case $c_1=-1$ and $c_2=1$. If $\cE$ is unstable, $\cE(n-1)$ satisfies the assumption of Proposition~\ref{lem:unsta-chern}. So $\cE$ splits. If $\cE$ is stable, the restriction $\cE|_{Q^5}$ is stable with $c_1=-1$ and $c_2=1$. Therefore by \cite[Main~Theorem(i)]{O90}, $\cE|_{Q^5}$ is a Cayley bundle. By \cite[Theorem~3.2]{O90}, no Cayley bundle extends to $Q^6$. Hence we have $n=5$. 
\end{proof}

\begin{proof}[Proof of Theorem~\ref{thm:main2}]
By Proposition~\ref{prop:unsta-sp}, Proposition~\ref{prop:nexist} and Proposition~\ref{prop:5-11}, the proof is  completed.
\end{proof}

\begin{remark}
For a rank 2 weak Fano bundle on $Q^4$, when $c_1(\mathcal{E}) = 0$, the classification can be done except for a few exceptions using methods similar to those in \cite{APW94}. However, it seems  that when $c_1(\mathcal{E}) = -1$, the situation becomes more complicated.
\end{remark}

\subsection*{Acknowledgement}The author is very grateful to his supervisor, Professor Kiwamu Watanabe, for helpful advice and various discussion for this paper. This work was supported by JST SPRING, Japan Grant Number JPMJSP2170.
\subsection*{Data availability}Data sharing not applicable to this article as no data sets were
generated or analyzed during the current study.
\bibliography{Main.bib}

\begin{thebibliography}{10}

\bibitem{APW94}
Vincenzo Ancona, and Thomas Peternell, and Jaroslaw A Wi{\'s}niewski.
\newblock Fano bundles and splitting theorems on projective spaces and quadrics.
\newblock {\em Pacific J.Math}, 163(1):17--42, 1994.




\bibitem{F83}
Klaus Fritzsche.
\newblock Linear-uniforme {B}\"{u}ndel auf {Q}uadriken.
\newblock {\em Ann. Scuola Norm. Sup. Pisa Cl. Sci. (4)}, 10(2):313--339, 1983.

\bibitem{FHI22}
Takeru Fukuoka, Wahei Hara and Daizo Ishikawa.
\newblock Classification of rank two weak {F}ano bundles on del {P}ezzo threefolds of degree four.
\newblock {\em Math. Z.}, 301(3):2883--2905, 2022.

\bibitem{FHI23}
Takeru Fukuoka, Wahei Hara and Daizo Ishikawa.
\newblock Rank two weak {F}ano bundles on del {P}ezzo threefolds of degree five.
\newblock {\em Internat. J. Math.}, 34(13):Paper No. 2350084, 2023.

\bibitem{F98}
William Fulton.
\newblock {\em Intersection theory}.
\newblock Springer-Verlag, Berlin, 1998.
\newblock Ergebnisse der Mathematik und ihrer Grenzgebiete. 3. Folge. A Series of Modern Surveys in Mathematics [Results in Mathematics and Related Areas. 3rd Series. A Series of Modern Surveys in Mathematics], No. 2.

\bibitem{H77}
Robin Hartshorne.
\newblock {\em Algebraic geometry}.
\newblock Springer-Verlag, New York-Heidelberg, 1977.
\newblock Graduate Texts in Mathematics, No. 52.

\bibitem{H78}
Robin Hartshorne.
\newblock Stable vector bundles of rank {$2$} on {${\bf P}^{3}$}.
\newblock {\em Math. Ann.}, 238(3):229--280, 1978.

\bibitem{I16}
Daizo Ishikawa.
\newblock Weak {F}ano bundles on the cubic threefold.
\newblock {\em Manuscr Math.}, 149(1-2):171--177, 2016.


\bibitem{L98}
Adrian Langer.
\newblock Fano {$4$}-folds with scroll structure.
\newblock {\em Nagoya Math. J.}, 150:135--176, 1998.

\bibitem{L041}
Robert Lazarsfeld.
\newblock {\em Positivity in algebraic geometry. {I}}, volume~48 of {\em
  Ergebnisse der Mathematik und ihrer Grenzgebiete. 3. Folge. A Series of
  Modern Surveys in Mathematics [Results in Mathematics and Related Areas. 3rd
  Series. A Series of Modern Surveys in Mathematics]}.
\newblock Springer-Verlag, Berlin, 2004.
\newblock Classical setting: line bundles and linear series.

\bibitem{L042}
Robert Lazarsfeld.
\newblock {\em Positivity in algebraic geometry. {II}}, volume~49 of {\em
  Ergebnisse der Mathematik und ihrer Grenzgebiete. 3. Folge. A Series of
  Modern Surveys in Mathematics [Results in Mathematics and Related Areas. 3rd
  Series. A Series of Modern Surveys in Mathematics]}.
\newblock Springer-Verlag, Berlin, 2004.
\newblock Positivity for vector bundles, and multiplier ideals.

\bibitem{L75}
Joseph. Le Potier.
\newblock Annulation de la cohomolgie \`a valeurs dans un fibr\'{e} vectoriel holomorphe positif de rang quelconque.
\newblock {\em Math. Ann.}, 218(1):35--53, 1975.

\bibitem{MP97}
Youichi Miyaoka and Thomas Peternell.
\newblock {\em Geometry of higher-dimensional algebraic varieties}, 
\newblock DMV Seminar, vol. 26, Birkh\"{a}user Verlag, Basel, 1997.

\bibitem{MOS12}
Roberto Mu{\~n}oz, Gianluca Occhetta and Luis~E. Sol{\'a}~Conde.
\newblock Rank two {F}ano bundles on {$\Bbb{G}(1,4)$}.
\newblock {\em J. Pure Appl. Algebra}, 216(10):2269--2273, 2012.

\bibitem{MOS14}
Roberto Mu{\~n}oz, Gianluca Occhetta and Luis~E. Sol{\'a}~Conde.
\newblock On rank 2 vector bundles on {F}ano manifolds.
\newblock {\em Kyoto J. Math.}, 54(1):167--197, 2014.

\bibitem{MOS142}
Roberto Mu{\~n}oz, Gianluca Occhetta and Luis~E. Sol{\'a}~Conde.
\newblock Classification theorem on {F}ano bundles.
\newblock {\em Ann. Inst. Fourier (Grenoble)}, 64(1):341--373, 2014.

\bibitem{O20}
Masahiro Ohno.
\newblock Nef vector bundles on a projective space with first {C}hern class three.
\newblock {\em Rend. Circ. Mat. Palermo (2)}, 69(2):425--458, 2020.


\bibitem{OSS80}
Christian Okonek, Michael Schneider, and Heinz Spindler.
\newblock {\em Vector bundles on complex projective spaces}.
\newblock Progress in Mathematics, 3. Birkh\"auser, Boston, Mass., 1980.


\bibitem{O90}
Giorgio Ottaviani.
\newblock On {C}ayley bundles on the five-dimensional quadric.
\newblock {\em Boll. Un. Mat. Ital. A (7)}, 4(1):87--100, 1990.

\bibitem{SSW91}
Ignacio Sols, Micha{\l} Szurek  and Jaroslaw A Wi{\'s}niewski.
\newblock Rank-{$2$} {F}ano bundles over a smooth quadric {$Q_3$}.
\newblock {\em Pacific J. Math.}, 148(1):153--159, 1991.

\bibitem{SW90}
Micha{\l} Szurek, and Jaroslaw A Wi{\'s}niewski.
\newblock Fano bundles over $\mathbb{P}^3$ and $Q^3$.
\newblock {\em Pacific J. Math}, 141(1):197--208, 1990.

\bibitem{V82}
Eckart Viehweg.
\newblock Vanishing theorems.
\newblock {\em J. Reine Angew. Math.}, 335:1--8, 1982.

\bibitem{Y12}
Kazunori Yasutake.
\newblock On the classification of rank 2 almost {F}ano bundles on projective space.
\newblock {\em Adv. Geom.}, 12(2):353--363, 2012.

\end{thebibliography}
\end{document}